\documentclass[12pt]{amsart}
\usepackage{latexsym,amssymb,amscd}
\def\Lc{{\mathcal{L}}}
\def\B'c{{\mathcal{B'}}}
\def\U'c{{\mathcal{U'}}}

\def\opn#1#2{\def#1{\operatorname{#2}}} 
\opn\chara{char}
\opn\length{\ell}
\opn\projdim{proj\,dim}
\opn\injdim{inj\,dim}
\opn\ini{in}
\opn\rank{rank}
\opn\Tiefe{Tiefe}
\opn\grade{grade}
\opn\height{height}
\opn\embdim{emb\,dim}
\opn\codim{codim}

\opn\Tr{Tr}
\opn\bigrank{big\,rank}
\opn\superheight{superheight}\opn\lcm{lcm}
\opn\trdeg{tr\,deg}%
\opn\reg{reg}
\opn\lreg{lreg}
\opn\deg{deg}
\opn\lcm{lcm}
\opn\div{div}
\opn\Div{Div}
\opn\cl{cl}
\opn\Cl{Cl}
\opn\Spec{Spec}
\opn\Supp{Supp}
\opn\supp{supp}
\opn\Sing{Sing}
\opn\Ass{Ass}
\opn\Ann{Ann}
\opn\Rad{Rad}
\opn\Soc{Soc}
\opn\Ker{Ker}
\opn\Coker{Coker}
\opn\Im{Im}
\opn\Hom{Hom}
\opn\Tor{Tor}
\opn\Ext{Ext}
\opn\End{End}
\opn\Aut{Aut}
\opn\id{id}

\opn\nat{nat}
\opn\GL{GL}
\opn\SL{SL}
\opn\mod{mod}
\opn\ord{ord}
\opn\depth{depth}
\opn\set{set}
\opn\Shad{Shad}
\opn\aff{aff}
\opn\con{conv}
\opn\relint{relint}
\opn\st{st}
\opn\lk{lk}
\opn\cn{cn}
\opn\core{core}
\opn\vol{vol}
\opn\gr{gr}

\def\pot#1#2{#1[\kern-0.28ex[#2]\kern-0.28ex]}

\opn\dirlim{\underrightarrow{\lim}}
\opn\invlim{\underleftarrow{\lim}}
\def\pnt{{\raise0.5mm\hbox{\large\bf.}}}

\def\Implies{\ifmmode\Longrightarrow \else
     \unskip${}\Longrightarrow{}$\ignorespaces\fi}
\def\implies{\ifmmode\Rightarrow \else
     \unskip${}\Rightarrow{}$\ignorespaces\fi}
\def\iff{\ifmmode\Longleftrightarrow \else
     \unskip${}\Longleftrightarrow{}$\ignorespaces\fi}

\let\:=\colon
\newtheorem{Theorem}{Theorem}[section]
\newtheorem{Lemma}[Theorem]{Lemma}
\newtheorem{Corollary}[Theorem]{Corollary}
\newtheorem{Proposition}[Theorem]{Proposition}
\newtheorem{Remark}[Theorem]{Remark}

\newtheorem{Definition}[Theorem]{Definition}

\let\epsilon=\varepsilon
\let\phi=\varphi
\let\kappa=\varkappa
\title{On the minimal graded free resolution of powers of lexsegment ideals}
\author{Anda Olteanu}
\thanks{The author was supported by the CNCS-UEFISCDI project PN II-RU PD 23/06.08.2010 and by the strategic grant POSDRU/89/1.5/S/58852, Project ``Postdoctoral program for training scientific researchers" co-financed by the European Social Fund within the Sectorial Operational Program Human Resources Development 2007 - 2013"}

\address{Faculty of Mathematics and Computer Science, Ovidius University, Bd.\ Mamaia 124,
 900527 Constanta, Romania,} \email{olteanuandageorgiana@gmail.com}

\begin{document}

\begin{abstract}  We consider powers of lexsegment ideals with a linear resolution (equivalently, with linear quotients) which are not completely lexsegment ideals. We give a complete description of their minimal graded free resolution.\\



\end{abstract}

\maketitle

\section*{Introduction}
Let $S=K[x_1,\ldots,x_n]$ be the polynomial ring in $n$ variables over a field $K$ and $<_{lex}$ be the lexicographical order with respect to $x_1>_{lex}\cdots>_{lex} x_n$. Fix an integer $d\geq2$ and let $u$ and $v$ be two monomials of degree $d$ in $S$ such that $u>_{lex}v$. The lexsegment ideal determined by the monomials $u$ and $v$, $(\mathcal{L}(u,v))$, is the monomial ideal generated by all the monomials $w$ in $S$ of degree $d$ which have the property that $u>_{lex}w>_{lex}v$.

Defined by Hulett and Martin \cite{HM}, lexsegment ideals have been studied in several papers \cite{ADH}, \cite{DH}, \cite{EOS}, \cite {EO}, \cite{I}. Their properties such as being Gotzmann, normally torsion-free or sequentially Cohen--Macaulay have been completely characterized \cite{OOS}, \cite{O}, \cite{I}. All the characterizations are in terms of the ends of the lexsegment. 

It is known that any ideal with linear quotients generated in one degree has a linear resolution, but the converse does not hold \cite{CH}. In \cite{EOS} it is proved that these two notions are equivalent for the class of lexsegment ideals. Moreover, for the case of completely lexsegment ideals with linear quotients, the minimal graded free resolution can be described. It is natural to ask whether the powers of an ideal with linear quotients have again linear quotients. Conca's example shows that this is not true in general \cite{C}, but for lexsegment ideals, this property is preserved by their powers, \cite{EO}.

We will consider powers of lexsegment ideals with a linear resolution which are not completely lexsegment ideal and we describe their minimal graded free resolution by proving that their decomposition function is regular and using the result of Herzog and Takayama for this case \cite{HT}. In this way, the minimal graded free resolution of lexsegment ideals with linear quotients is completely described.

The paper is organized in three sections. In the first section, we fix all the notations and the terminology and we recall some known results which will play a key role in the proofs.

In the second section, we consider powers of a lexsegment ideal $I$ with linear quotients which is not a completely lexsegment ideal. We describe the decomposition function associated to the increase reverse lexicographical order and we show that this is regular. By using the results of Herzog and Takayama \cite{HT}, we may write the minimal graded free resolution of $I^k$, for all $k\geq 1$.

In the last section we consider an example in order to illustrate the results.

\section{Preliminaries}
Let $S=K[x_1,\ldots,x_n]$ be the polynomial ring in $n$ variables over a field $K$ and we fix the lexicographical order, $<_{lex}$, on $S$ with respect to the order of the variables $x_1>_{lex}\cdots>_{lex}x_n$. For a monomial $m=x_1^{a_1}\cdots x_n^{a_n}$, we denote by $\nu_i(m)$ the exponent of the variable $x_i$ in the monomial $m$, that is $\nu_i(m)=a_i$. The set $\supp(m)=\{i\ :\ \nu_i(m)\neq 0\}$ is called the \it{support }\rm of the monomial $m$. Let us denote $\min(m):=\min(\supp(m))$ and $\max(m):=\max(\supp(m))$. If $I$ is a monomial ideal in $S$, then $G(I)$ will be the set of its minimal monomial generators.

For $d\geq 2$ an integer, we denote by $\mathcal{M}_d(S)$ the set of all the monomials of degree $d$ in $S$. Let $u,v\in\mathcal{M}_d(S)$ be two monomials such that $u>_{lex}v$. The set $$\mathcal{L}(u,v)=\{w\in\mathcal{M}_d(S)\ :\ u>_{lex}w>_{lex} v\}$$ is called \it{the lexsegment set }\rm determined by the monomials $u$ and $v$. A \it{lexsegment ideal }\rm is a monomial ideal generated by a lexsegment set. An important notion in the study of the lexsegment ideals is the shadow of a set of monomials. For a set of monomials $T\subseteq S$, one may define its \it{shadow }\rm as being the set $\Shad(T)=\{x_iw\ :\ 1\leq i\leq n,\ w\in T\}$. Moreover, the $i$-th shadow is recursively defined as $\Shad^i(T)=\Shad^{i-1}(\Shad(T))$. 
 
A lexsegment set is a \textit{completely lexsegment set} if all the iterated shadows are again lexsegment sets. An ideal generated by a completely lexsegment set is called a \textit{completely lexsegment ideal}. 

In \cite{HT}, is considered the class of ideals with linear quotients. We recall the definition for the particular class of monomial ideals.
\begin{Definition}\cite{HT}\rm{\ } A monomial ideal $I\subseteq S$ \textit{has linear quotients} if there exists an ordering of its minimal monomial generators $m_1,\ldots,m_r$ such that the ideal $(m_1,\ldots,m_{i-1}):(m_i)$ is generated by a set of variables, for all $i\geq 2$.
\end{Definition}

 If $I$ is a monomial ideal which has linear quotients with respect to the sequence $m_1,\ldots,m_r$, then one may consider the sets $$\set(m_i)=\{j\ :\ x_j\in(m_1,\ldots,m_{i-1}):(m_i)\},$$ for all $i\geq2$. 
 
The following result collects known results on lexsegment ideals.

\begin{Theorem}[\cite{ADH}, \cite{EOS}, \cite{EO}]
Let $u=x_1^{a_1}\cdots x_{n}^{a_n}$ with $a_1>0$ and $v=x_1^{b_1}\cdots x_n^{b_n}$ be monomials of degree $d$ with $u\geq_{lex}v$ and let $I=(\mathcal{L}(u,v))$ 
be a  lexsegment ideal. Then the following statements are equivalent;
\begin{itemize}
\item [(1)] $I$ has a linear resolution.
\item [(2)] $I$ has linear quotients.
\item [(3)] All the powers of $I$ have linear quotients.
\item [(4)] All the powers of $I$ have a linear resolution.
\end{itemize}
\end{Theorem}
If we restrict to the case of lexsegment ideals which are not completely lexsegment, we have the following result which combines \cite[Theorem 2.4]{ADH}, \cite[Theorem 2.1]{EOS}, \cite[Corollary 3.9]{EO}:

\begin{Theorem}\label{thnoncomp}
Let $u=x_1^{a_1}\cdots x_{n}^{a_n}$ with $a_1>0$ and $v=x_2^{b_2}\cdots x_n^{b_n}$ be monomials of degree $d$ with $u\geq_{lex}v$, and let $I=(\mathcal{L}(u,v))$ be a lexsegment ideal which is note completely lexsegment. Then the following statements are equivalent;
\begin{itemize}
\item [(1)] $u$ and $v$ have the following form:
	\[u=x_1x_{l+1}^{a_{l+1}}\cdots x_n^{a_n}\ \mbox{and } v=x_lx_{n}^{d-1}
\]
for some $l$, $2\leq l\leq n-1$.
\item [(2)] $I$ has a linear resolution.
\item [(3)] $I$ has linear quotients.
\item [(4)] All the powers of $I$ have linear quotients.
\item [(5)] All the powers of $I$ have a linear resolution.
\end{itemize}
\end{Theorem}

The order of the minimal monomial generators for which $I^k$ has linear quotients for all $k\geq1$, where $I$ is a lexsegment ideal with a linear resolution which is not completely lexsegment, is the increasing reverse lexicographical order. We recall that $m_1<_{revlex}m_2$ if there is some $s$, $1\leq s\leq n$, such that $\nu_i(m_1)=\nu_i(m_2)$ for all $i\geq s$ and $\nu_s(m_1)>\nu_s(m_2)$.

\begin{Remark}\rm$\ $ Let $u,v\in \mathcal{M}_d$ be two monomials, $u\geq_{lex}v$, and $I=(\mathcal{L}(u,v))$ be the corresponding lexsegment ideal.
We note that we may always assume that $x_1\mid u$ and $x_1\nmid v$. Indeed, if $x_1\mid v$ we denote $u=x_1^{a_1}\cdots x_n^{a_n}$ and $v=x_1^{b_1}\cdots x_n^{b_n}$, with $a_1\geq b_1>0$. If $a_1=b_1$, then $I=(\mathcal{L}(u,v))$ is isomorphic, as an $S$-module, to the ideal generated by the lexsegment $\mathcal{L}(u/x_1^{a_1},v/x_1^{b_1})$ of degree $d-a_1$. This lexsegment may be studied in the polynomial ring in a smaller number of variables. If $a_1>b_1$, then $I=(\mathcal{L}(u,v))$ and $(\mathcal{L}(u/x_1^{b_1},v/x_1^{b_1}))$ are isomorphic as $S$-modules and we have $\nu_1(u/x_1^{b_1})>1$ and $\nu_1(v/x_1^{b_1})=0$.  Therefore we will always assume that $x_1\mid u$ and $x_1\nmid v$.
\end{Remark}
\section{Powers of lexsegment ideals with a linear resolution which are not completely lexsegments}

In the sequel, we show that all the powers of lexsegment ideals with a linear resolution which are not completely lexsegment ideals have regular decomposition function with respect to the increasing reverse lexicographical order.
For two monomials $u,v$ of degree $d$, we denote by $\Lc(u,v)$ the corresponding lexsegment ideal. We will consider only the case when $x_1\mid u$ and $x_1\nmid v$.

By using Theorem~\ref{thnoncomp}, we will assume that $u$ and $v$ are monomials of degree $d\geq2$ such that $I=(\Lc(u,v))$ is a lexsegment ideal which is not a completely lexsegment ideal, and $u$ and $v$ have the following form:
	\[u=x_1x_{l+1}^{a_{l+1}}\cdots x_n\ \mbox{and } v=x_lx_{n}^{d-1}
\]
for some $l$, $2\leq l\leq n-1$.

For a lexsegment $\mathcal{L}(u,v)$, we assume that the elements are ordered by the increasing reverse lexicographical order. We denote by $I=(\Lc(u,v))$ the lexsegment ideal, and by $I^k_{<_{revlex} w}$, the ideal generated by all the monomials $z\in G(I^k)$ with $z<_{revlex} w$. $I^k_{\leq_{revlex}w}$ will be the ideal generated by all the monomials $z\in G(I^k)$ with $z\leq_{revlex} w$. 

\begin{Remark}\rm\label{set} If $m\in G(I^k)$ and $s\in\set(m)$, then there exists a monomial $w\in G(I^k)$, $w<_{revlex}m$ such that $x_sm=x_tw$, for some $t$, $1\leq t\leq n$. Since $m\neq w$, we must have $s\neq t$ and $x_t\mid m$. Moreover, $w=x_sm/x_t m<_{revlex}m$ implies that $s>t$.
\end{Remark}

In order to describe the decomposition function, we need some preparatory results.

\begin{Lemma}\label{setmin} Let $I=(\mathcal{L}(u,v))\subset S$ be a lexsegment ideal with a linear resolution which is not a completely lexsegment and $m\in G(I^k)$ a monomial. If $s\in\set(m)$, then $s>\min(m)$.
\end{Lemma}

\begin{proof} Since $s\in\set(m)$, by using the above remark, we have that $x_sm=wx_t$, for some $w\in G(I^k)$, $w<_{revlex}m$, and some $t$, $1\leq t\leq n$. Moreover, $s>t$. The statement follows, since $x_t\mid m$ implies that $t\geq\min(m)$.
\end{proof}

One may note that, once we fix an integer $l$, $2\leq l\leq n-1$, a monomial $m\in S$ may be uniquely written as $m= \overline m\tilde m$, with $ \overline{m}\in K[x_1,\ldots,x_l]$ and $\tilde{m}\in K[x_{l+1},\ldots,x_n]$. In particular, we have that $\max( \overline{m})\leq l<\min(\tilde{m})$. On the set of all the monomials of degree $kd$ in $S$, $\mathcal{M}_{kd}(S)$, we define the order $\prec$ as follows: for $m_1,m_2\in\mathcal{M}_{kd}(S)$, we say that $m_1\prec m_2$ if $\deg( \overline{m_1})<\deg( \overline{m_2})$ or $\deg( \overline{m_1})=\deg( \overline{m_2})$ and $m_1<_{lex}m_2$.

If $I=(\mathcal{L}(u,v))$, with $x_1\mid u$ and $x_1\nmid v$, is a lexsegment ideal with a linear resolution which is not a completely lexsegment, then $u=x_1x_{l+1}^{a_l+1}\cdots x_n^{a_n}$ and $v=x_lx_n^{d-1}$, for some integer $l$, $2\leq l\leq n-1$. Therefore, through this paper, we will assume that the fixed integer which will be used in the order $\prec$ is $l$.

\begin{Remark}\rm  If $m\in G(I^k)$, then $\deg( \overline{m})\geq k$, since $u=x_1x_{l+1}^{a_{l+1}}\cdots x_n^{a_n}\ \mbox{and } v=x_lx_{n}^{d-1}$, for some $l$, $2\leq l\leq n-1$.
\end{Remark}

\begin{Lemma}\label{setmin2} Let $I=(\mathcal{L}(u,v))\subset S$ be a lexsegment ideal with a linear resolution which is not a completely lexsegment ideal and $m\in G(I^k)$ a monomial. If $s\in\set(m)$ and $x_sm/x_{\min(m)}\prec v^k$, then $s>\min(\tilde{m})$.
\end{Lemma}

\begin{proof}  By the hypothesis we have $x_sm/x_{\min(m)}\prec v^k$. Writing $m$ as $m= \overline m\tilde m$, we get that the only possible case is that when $\deg(\overline{x_sm/x_{\min(m)}})<\deg(\overline{v^k})=\deg(\overline{x_l^kx_n^{k(d-1)}})=k$. Indeed, if we assume that $\deg(\overline{x_sm/x_{\min(m)}})=\deg(\overline{v^k})=k$, then $x_sm/x_{\min(m)}<_{lex}v^k=x_l^kx_n^{k(d-1)}$. In particular, $x_sm/x_{\min(m)}\leq_{lex}x_l^{k-1}x_{l+1}^{k(d-1)+1}$, a contradiction. Therefore, we have that $\deg(\overline{x_sm/x_{\min(m)}})<k$ which implies that $\deg(\overline{m})=k$ and $s>l$. 

Since $s\in \set(m)$, according to Remark~\ref{set}, we have $x_sm=x_tw$, for some $w\in G(I^k)$, $w<_{revlex}m$, for some $t\in\{1,\ldots, n\}$ and $s>t$. One may note that, since $w\in G(I^k)$ and $x_t\mid m$, we must have $t\geq\min(\tilde{m})$ because otherwise we will get that $w=x_sm/x_t$ has $\deg(\overline{m})=k-1$, which is impossible. 
\end{proof}

In \cite{HT}, J. Herzog and Y. Takayama defined the decomposition function of a monomial ideal with linear quotients. We recall their definition. 

\begin{Definition}\cite{HT} \rm Let $I\subset S$ be a monomial ideal with linear quotients with respect to the sequence of minimal monomial generators $u_1,\ldots, u_m$ and set $I_j=(u_1,\ldots, u_j)$, for $j=1,\ldots,m$. Let $M(I)$ be the set of all monomials in $I$. The map $g:M(I)\rightarrow G(I)$ defined as: $g(u)=u_j$, where $j$ is the smallest number such that $u\in I_j$, is called \it the decomposition function \rm of $I$.
\end{Definition}

By using the above results, we may completely describe the decomposition function associated to the increasing reverse lexicographical order. Note that, since $I$ is a lexsegment ideal with a linear resolution which is not a completely lexsegment, then $I$ has linear quotients with respect to the increasing reverse lexicographical order. Moreover, $I^k$ has linear quotients, for all $k\geq 1$, by \cite[Corollary 3.9]{EO}.

\begin{Proposition}\label{dec1} Let $I=(\mathcal{L}(u,v))\subset S$ be a lexsegment ideal with a linear resolution which is not a completely lexsegment ideal and $g:M(I^k)\rightarrow G(I^k)$ the decomposition function with respect to the increasing reverse lexicographical order. If $m\in G(I^k)$ and $s\in\set(m)$ such that $x_sm/x_{\min(m)}\succeq v^k$, then $g(x_sm)=x_sm/x_{\min(m)}$.
\end{Proposition}
\begin{proof} Let $m\in G(I^k)$ and $s\in\set(m)$. We have to show that $x_sm/x_{\min(m)}\in G(I^k)$ and 	\[\frac{x_{s}m}{x_{\min(m)}}=\min\ _{revlex}\{w\in G(I^k)\ :\ w<_{revlex}m,\ x_{s}m\in I^k_{\leq_{revlex}w}\}.
\]
 If $x_sm/x_{\min(m)}=v^k$, then it is obvious that $x_sm/x_{\min(m)}\in G(I^k)$. Let us assume that $x_sm/x_{\min(m)}\succ v^k$.  By Lemma~\ref{setmin}, we have that $s>\min(m)$. The fact that $x_sm/x_{\min(m)}\succ v^k$ implies one of the following $\deg(\overline{x_sm/x_{\min(m)}})>\deg(\overline{v^k})=k$ or $\deg(\overline{x_sm/x_{\min(m)}})=\deg(\overline{v^k})=k$ and $x_sm/x_{\min(m)}>_{lex}v^k$. 

In order to show that $x_sm/x_{\min(m)}\in G(I^k)$, we split the proof in two cases due to the discussions involved by $x_sm/x_{\min(m)}\succ v^k$:

\textit{Case I:} We assume that $\deg(\overline{x_sm/x_{\min(m)}})>\deg(\overline{v^k})=k$. Since $m\in G(I^k)$ there exist $m_1,\ldots, m_k\in\mathcal{L}(u,v)$, such that $m=m_1\cdots m_k$. Let $1\leq i\leq k$ be such that $\min(m)=\min(m_i)$. Then
	\[\frac{x_sm}{x_{\min(m)}}=x_s m_1\cdots m_{i-1}\frac{m_i}{x_{\min(m_i)}}m_{i+1}\cdots m_k\geq v^k.
\]
If $x_sm_i/x_{\min(m_i)}\in \Lc(u,v)$, then we are finished. We assume that $x_{s}m_i/x_{\min(m_i)}\notin\Lc(u,v)$, that is $x_sm_i/x_{\min(m_i)}<_{lex} v=x_lx_n^{d-1}$, since $s>\min(m_i)=\min(m)$. In particular, $\supp(x_sm_i/x_{\min(m_i)})\subseteq\{l+1,\ldots,n\}$ and $s\geq l+1$. Since $\deg(\overline{x_sm/x_{\min(m)}})>k$, there exist $1\leq j,r\leq l$ and $1\leq \alpha\leq k$ such that $x_jx_r\mid m_{\alpha}$. In particular, we must have $j,r\geq2$ by using the form of the monomials $u$ and $v$. Then
 \[\frac{x_sm}{x_{\min(m)}}=m_1\cdots \frac{x_sm_{\alpha}}{x_j}\cdots \frac{x_jm_i}{x_{\min(m_i)}}\cdots m_k\geq_{lex} v^k
\]
where $v\leq_{lex}x_jm_i/x_{\min(m_i)}\leq_{lex} m_i\leq_{lex} u$ and $v\leq_{lex}x_sm_{\alpha}/x_{j}<_{lex} m_{\alpha}\leq_{lex} u$. This implies $x_sm/x_{\min(m)}\in G(I^k)$.

\textit{Case II:} We assume that $\deg(\overline{x_sm/x_{\min(m)}})=\deg(\overline{v^k})=k$, therefore we must have $x_sm/x_{\min(m)}>_{lex}v^k$. Since $\deg(\overline{x_sm/x_{\min(m)}})=k$, we can have that $s\leq l$ or $s>l$ and $\deg(\overline{m})=k+1$.

Since $m\in G(I^k)$ there exist $m_1,\ldots, m_k\in\mathcal{L}(u,v)$, such that $m=m_1\cdots m_k$. Let $1\leq i\leq n$ be such that $\min(m)=\min(m_i)$.

If $s\leq l$, then, since $m=m_1\cdots m_k$ and using above the notations, we get

	\[\frac{x_sm}{x_{\min(m)}}= m_1\cdots m_{i-1}\frac{x_sm_i}{x_{\min(m_i)}}m_{i+1}\cdots m_k\geq v^k\in G(I^k)
\]
because $\min(m)=\min(m_i)<s\leq l$.

If $s>l$, then $\deg(\overline{m})=k+1$ and, as in the Case I, there exist $1\leq j,r\leq l$ and $1\leq \alpha\leq k$ such that $x_jx_r\mid m_{\alpha}$. In particular, we must have $j,r\geq2$ by using the form of the monomials $u$ and $v$. Then
 \[\frac{x_sm}{x_{\min(m)}}=m_1\cdots \frac{x_sm_{\alpha}}{x_j}\cdots \frac{x_jm_i}{x_{\min(m_i)}}\cdots m_k\geq_{lex} v^k
\]
where $v\leq_{lex}x_jm_i/x_{\min(m_i)}\leq_{lex} m_i\leq_{lex} u$ and $v\leq_{lex}x_sm_{\alpha}/x_{j}<_{lex} m_{\alpha}\leq_{lex} u$. This implies $x_sm/x_{\min(m)}\in G(I^k)$.

Therefore, we proved that $x_sm/x_{\min(m)}\in G(I^k)$. 

We have to prove that
	\[\frac{x_{s}m}{x_{\min(m)}}=\min\ _{revlex}\{w\in G(I^k)\ :\ w<_{revlex}m,\ x_{s}m\in I^k_{\leq_{revlex}w}\}.
\]
Let $w\in G(I^k)$ be such that $w<_{revlex} m$ and $x_sm\in I^k_{\leq_{revlex}w}$.
Since $x_sm\in I^k_{\leq_{revlex}w}$, there exists $w_1\in G(I^k)$, $w_1\leq_{revlex}w$, such that $x_sm=x_{t}w_1$, for some $t$, $1\leq t\leq n$. The fact that $m\neq w_1$ implies $s\neq t$. Hence, we must have $x_{t}\mid m$, in particular $t\geq\min(m)$. Therefore $$w\geq_{revlex}w_1=\frac{x_sm}{x_{t}}\geq_{revlex}\frac{x_sm}{x_{\min(m)}}$$ as desired.
\end{proof}

\begin{Proposition}\label{dec2} Let $I=(\mathcal{L}(u,v))\subset S$ be a lexsegment ideal with a linear resolution which is not a completely lexsegment ideal and $g:M(I^k)\rightarrow G(I^k)$ the decomposition function with respect to the increasing reverse lexicographical order. If $m\in G(I^k)$ and $s\in\set(m)$ such that $x_sm/x_{\min(m)}\prec v^k$, then $g(x_sm)=x_sm/x_{\min(\tilde{m})}$.
\end{Proposition}
\begin{proof} One may easy see that we may only have $\deg(\overline{x_sm/x_{\min(m)}})<k$ which implies that $\deg(\overline{m})=k$. Indeed, if $\deg(\overline{x_sm/x_{\min(m)}})=k$, then $x_sm/x_{\min(m)}<_{lex} v^k$, that is $x_sm/x_{\min(m)}\leq_{lex} x_l^{k-1}x_{l+1}^{k(d-1)+1}$ which is impossible since each monomial $w$ of this form has $\deg(\overline{w})<k$. By Lemma~\ref{setmin2}, we have $s>\min(\tilde{m})>l$. 

Firstly, we prove that $x_sm/x_{\min(\tilde{m})}\in G(I^k)$. Since $m\in G(I^k)$ there exist $m_1,\ldots, m_k\in\mathcal{L}(u,v)$, such that $m=m_1\cdots m_k$.  Let $1\leq i\leq k$ be such that $x_{\min(\tilde{m})}\mid m_i$. Then
	\[\frac{x_sm}{x_{\min(\tilde{m})}}=m_1\cdots m_{i-1}\frac{x_sm_i}{x_{\min(m'_i)}}\cdots m_k\in G(I^k)
\]
since $x_sm_i/x_{\min(m'_i)}\in\Lc(u,v)$ because $s>\min(\tilde{m})\geq l+1$.

Next, we have to prove that
	\[\frac{x_{s}m}{x_{\min(\tilde{m})}}=\min\ _{revlex}\{w\in G(I^k)\ :\ w<_{revlex}m,\ x_{s}m\in I^k_{\leq_{revlex}w}\}.
\]
Let $w\in G(I^k)$ be such that $w<_{revlex} m$ and $x_sm\in I^k_{\leq_{revlex}w}$ which implies that there exists $w_1\in G(I^k)$, $w_1\leq_{revlex}w$ such that $x_sm=x_{t}w_1$, for some $t$, $1\leq t\leq n$. The fact that $m\neq w_1$ implies $s\neq t$. Hence, we must have $x_{t}\mid m$, in particular $t\geq\min(m)$. Since $\deg(\overline{m})=k$, $s>\min(\tilde{m})>l$, and $w\in G(I^k)$, we must have that $\deg(\overline{w})=k$ which implies that $t\geq\min(\tilde{m})$, since $x_t\mid m$. Therefore $w_1=x_sm/x_t\geq_{revlex}x_sm/x_{\min(\tilde{m})}$, which ends the proof.
\end{proof}

Propositions~\ref{dec1} and \ref{dec2} give us a complete description of the decomposition function.

\begin{Corollary}\label{dec} Let $I=(\mathcal{L}(u,v))\subset S$ be a lexsegment ideal with a linear resolution which is not a completely lexsegment ideal and $g:M(I^k)\rightarrow G(I^k)$ the decomposition function with respect to the increasing reverse lexicographical order. Then
	\[g(x_sm)=
\left\{\begin{array}{cc}
x_sm/x_{\min(m)},&x_sm/x_{\min(m)}\succeq v^k\\
x_sm/x_{\min(\tilde{m})},&x_sm/x_{\min(m)}\prec v^k.
\end{array}\right.
\]
for any $m\in G(I^k)$, $s\in\set(m)$, and $m= \overline m\tilde m$.
\end{Corollary}

Let $I$ be a monomial ideal with linear quotients. We say that the decomposition function $g:M(I)\rightarrow G(I)$ associated to the corresponding order of monomials is \it regular \rm if $\set(g(x_su))$$\subseteq\set(u)$ for all $s\in\set(u)$ and $u\in G(I)$. In the sequel, we show that, for the powers of lexsegment ideals $I$ with a linear resolution which are not completely lexsegment, the decomposition function $g:M(I^k)\rightarrow G(I^k)$ associated to the increasing reverse lexicographical order of the generators from $G(I^k)$ is regular.

\begin{Proposition}\label{reg1} Let $I=(\mathcal{L}(u,v))\subset S$ be a lexsegment ideal with a linear resolution which is not a completely lexsegment ideal and $g:M(I^k)\rightarrow G(I^k)$ the decomposition function with respect to the increasing reverse lexicographical order. Let $m\in G(I^k)$ and $s\in\set(m)$ be such that $x_sm/x_{\min(m)}\succ v^k$ and let $t\in\set(g(x_sm))$. Then $t\in\set(m)$.
\end{Proposition}
\begin{proof} By Lemma~\ref{setmin}, we have that $s>\min(m)$.
By the hypothesis, $x_sm/x_{\min(m)}\succ v^k$, therefore, by Proposition~\ref{dec}, we have $g(x_sm)=x_sm/x_{\min(m)}=w_1$. Since $t\in\set(w_1)$, we get $x_tw_1\in I^k_{<_{revlex}w_1}$. Hence there exist $w\in G(I^k)$, $w<_{revlex}w_1$, and $1\leq j\leq n$ such that $x_tw_1=x_jw$, that is
	\[x_tx_sm=x_jx_{\min(m)}w.
\]
One may note that $j\neq t$ (otherwise $w=w_1$, contradiction), hence $x_j\mid x_sm$. Since $t\in\set(w_1)$ and using Lemma~\ref{setmin} we obtain that $t>\min(w_1)\geq\min(m)$. 

If $j=s$, then $x_tm=x_{\min(m)}w$ and $t\in\set(m)$. 

Let us assume that $j\neq s$. We show that $x_{\min(m)}w/x_s\in G(I^k)$. We write $m=m_1\cdots m_k$, with $m_1,\ldots,m_k\in\Lc(u,v)$. Let $1\leq i\leq k$ be such that $x_j\mid m_i$. Now, the fact that $w<_{revlex}w_1$ implies that $x_{\min(m)}w<_{revlex}x_{\min(m)}w_1=x_sm$. Therefore $x_{\min(m)}w/x_s<_{revlex}m$ and, taking into account that $x_{\min(m)}w/x_s=x_tm/x_j$, we get $x_tm/x_j<_{revlex}m$, that is $t>j$. 

If $\deg(\overline{m})>k$ then there exist $1\leq p\leq k$ and $1\leq \alpha\leq l$ such that $\deg(\overline{m_p})\geq2$ and $x_{\alpha}\mid m_p$. In particular, we must have $\alpha\geq 2$ (by using the form of the monomials $u$ and $v$). In this case
	\[\frac{x_tm}{x_j}=m_1\cdots \frac{x_{\alpha}m_i}{x_j}\cdots\frac{x_tm_p}{x_\alpha}\cdots m_k.
\] 
which implies $x_{t}m/x_{j}\in G(I^k)$ since $x_{\alpha}m_i/x_j$ and $x_tm_p/x_{\alpha}$ belong to $\Lc(u,v)$.

Let us assume that $\deg(\overline{m})=k$. Since $j<t$, we get $x_{t}m/x_{j}<_{lex}m$. If $\deg(\overline{x_tm/x_j})=k$, then we obviously have $x_tm/x_j\in G(I^k)$. We assume that $\deg(\overline{x_tm/x_j})=k-1$, that is $j\leq l$ and $t>l$. We also have $\min(m)\leq l$. Hence, $\deg(\overline{m})=k$ and the equality $x_tx_sm=x_jx_{\min(m)}w$ imply
	\[k\leq\deg(\overline{w})=\deg(\overline{m})+\nu_1(x_tx_s)+\cdots+\nu_l(x_tx_s)-2,
\]
which yields $\nu_1(x_tx_s)+\cdots+\nu_l(x_tx_s)=2$, that is $t,s\leq l$, a contradiction.

We proved that $x_{\min(m)}w/x_s<_{revlex}m$ and $x_{\min(m)}w/x_s\in G(I^k)$, hence $t\in\set(m)$.
\end{proof}

\begin{Proposition}\label{reg2} Let $I=(\mathcal{L}(u,v))\subset S$ be a lexsegment ideal with a linear resolution which is not a completely lexsegment ideal and $g:M(I^k)\rightarrow G(I^k)$ the decomposition function with respect to the increasing reverse lexicographical order. Let $m\in G(I^k)$ and $s\in\set(m)$ be such that $x_sm/x_{\min(m)}\prec v^k$ and let $t\in\set(g(x_sm))$. Then $t\in\set(m)$.
\end{Proposition}
\begin{proof} In the case when $x_sm/x_{\min(m)}\neq v^k$, by Proposition~\ref{dec} we have $g(x_sm)=x_sm/x_{\min(\tilde{m})}=w_1$. Since $t\in\set(w_1)$, we get $x_tw_1\in I^k_{<_{revlex}w_1}$. Hence, by the Remark~\ref{set}, $x_tw_1=x_jw$, for some  $w\in G(I^k)$, $w<_{revlex}w_1$, and $t>j\geq \min(w_1)$. Therefore, we get that
	\[x_tx_sm=x_jx_{\min(\tilde{m})}w.
\]
Also, one may note that the only possible case is that in which $\deg(\overline{x_sm/x_{\min(m)}})<k$. By Lemma~\ref{setmin2}, we have $s>\min(\tilde{m})$. It is easily seen that $j\neq t$ (otherwise $w=w_1$, contradiction). 

If $j=s$, then $x_tm=x_{\min(\tilde{m})}w$ and $t\in\set(m)$.

We assume now that $j\neq s$. Since $s>\min(\tilde{m})>l$ and $w\in G(I^k)$, one may easy note that $x_{\min(\tilde{m})}w/x_s=x_tm/x_j\in G(I^k)$ by the form of the monomials $u$ and $v$, excepting the case when $\overline{w}=x_1^k$. But in this case, $x_1^k\mid w_1$, therefore $x_1^k\mid m$. Thus we must have that $j>l$, which implies that $x_tm/x_j\in G(I^k)$, therefore $x_{\min(\tilde{m})}w/x_s\in G(I^k)$. Moreover, $x_{\min(\tilde{m})}w/x_s=x_tm/x_j<_{revlex}m$, hence $t\in\set(m)$.
\end{proof}

\begin{Proposition}\label{reg3} Let $I=(\mathcal{L}(u,v))\subset S$ be a lexsegment ideal with a linear resolution which is not a completely lexsegment ideal and $g:M(I^k)\rightarrow G(I^k)$ the decomposition function with respect to the increasing reverse lexicographical order. Let $m\in G(I^k)$ and $s\in\set(m)$ be such that $x_sm/x_{\min(m)}=v^k$ and let $t\in\set(g(x_sm))$. Then $t\in\set(m)$.
\end{Proposition}
\begin{proof} In this case, one may easy note that, we can have either $s\leq l$, which implies in fact that $s=l$, or $s>l$ and $\deg(\overline{m})=k+1$.

By Proposition~\ref{dec} we have $g(x_sm)=x_sm/x_{\min(m)}=v^k=w_1$. Since $t\in\set(w_1)$, we get $x_tw_1\in I^k_{<_{revlex}w_1}$. Hence, by the Remark~\ref{set}, $x_tw_1=x_jw$, for some  $w\in G(I^k)$, $w<_{revlex}w_1$, and $t>j\geq \min(w_1)=l$. Note that $\deg(\overline{w})\leq \deg(\overline{w_1})=k$, which implies $\deg(\overline{w})=k$. Therefore, we get that
	\[x_tx_sm=x_jx_{\min(m)}w.
\]
If $j=s$, then $x_tm=x_{n}w$ and $t\in\set(m)$. Therefore, we assume that $j\neq s$.

The case $s=l$ is impossible. Indeed, if $s=l$, then we must have $j>l$ since $j\geq \min(w_1)=l$ and $j\neq s$. Thus $j=n$ since $x_j\mid w_1=v^k$. But this is a contradiction since $t\neq j$.

If $s>l$, then $s=n$. In this case $\deg(\overline{m})=k+1$ which implies that $\deg(\overline{w})=k$ and $l<j<n$. Therefore $x_jw/x_s\in G(I^k)$. Thus $x_tm=x_n(x_jw/x_s)$ and $t\in \set(m)$.  
\end{proof}
By combining Propositions \ref{reg1}, \ref{reg2}, and \ref{reg3} we obtain:

\begin{Theorem}\label{reg} Let $I=(\mathcal{L}(u,v))\subseteq S$ be a lexsegment ideal generated in degree $d>1$ with a linear resolution which is not a completely lexsegment ideal. Then the decomposition function $g:M(I^k)\rightarrow G(I^k)$ associated to the increasing reverse lexicographical order of the generators from $G(I^k)$ is regular.
\end{Theorem}

By using the decomposition function, one may completely describe the resolution as J. Herzog and Y. Takayama showed, \cite{HT}.

\begin{Lemma}\cite{HT} Suppose $\deg\ u_1 \leq \deg\ u_2 \leq \cdots\leq \deg\ u_m.$ Then the iterated mapping
cone $\mathbb{F}$, derived from the sequence $u_1,\ldots,u_m,$ is a minimal graded free resolution
of $S/I$, and for all $i > 0$ the symbols
\[f(\sigma; u)\ \mbox{with}\ u\in G(I),\ \sigma \subset \set(u),\  |\sigma| = i - 1
\]
form a homogeneous basis of the $S-$module $F_i$. Moreover $\deg(f(\sigma; u)) = |\sigma| +\deg(u)$.
\end{Lemma} 

\begin{Theorem}\cite{HT} Let $I$ be a monomial ideal of $S$ with linear quotients, and $\mathbb{F}_{\bullet}$ the graded
minimal free resolution of $S/I$. Suppose that the decomposition function $g : M(I) \rightarrow G(I)$ is regular. Then the chain map $\partial$ of $\mathbb{F}_{\bullet}$ is given by
\[\partial(f(\sigma; u)) = -\sum_{s\in\sigma}(-1)^{\alpha(\sigma;s)}x_sf(\sigma\setminus s;u)+\sum_{s\in\sigma}(-1)^{\alpha(\sigma;s)}\frac{x_su}{g(x_su)}f(\sigma\setminus s;g(x_su)),\]
if $\sigma\neq\emptyset$, and
\[\partial(f(\emptyset; u)) = u\] otherwise.
Here $\alpha(\sigma;s)=|\{t\in\sigma\ |\ t<s\}|$.
\end{Theorem}

In our specific context we get the following
\begin{Corollary} Let $I=(\mathcal{L}(u,v))\subset S$ be a lexsegment ideal with linear quotients with respect to increasing reverse lexicographical order which is not a completely lesegment ideal and $\mathbb{F}_{\bullet}$ the graded minimal free resolution of $S/I^k$. Then the chain map of $\mathbb{F}_{\bullet}$ is given by
\[\partial(f(\sigma; w)) = \sum_{\stackrel{s\in\sigma:}{x_sw/x_{\min(w)}\succeq v^k}}(-1)^{\alpha(\sigma;s)}x_{\min(w)}f\left(\sigma\setminus s;\frac{x_sw}{x_{\min(w)}}\right)+\]\[+\sum_{\stackrel{s\in\sigma:}{x_sw/x_{\min(w)}\prec v^k}}(-1)^{\alpha(\sigma;s)}x_{\min(\tilde{w})}f\left(\sigma\setminus s;\frac{x_sw}{x_{\min(\tilde{w})}}\right)-\sum_{s\in\sigma}(-1)^{\alpha(\sigma;s)}x_sf(\sigma\setminus s;w),\]
if $\sigma\neq\emptyset$, and
\[\partial(f(\emptyset; w)) = w\] otherwise. For convenience we set $f(\sigma;w)=0$ if $\sigma\nsubseteq\set{w}$.
\end{Corollary}
\section{ An example}
Let $u=x_1x_3$ and $v=x_2x_4$  be monomials in the polynomial ring $S=k[x_1,x_2,x_3,x_4]$. Then $$\mathcal{L}(u,v)=\{x_1x_3,\ x_1x_4,\ x_2^2,\ x_2x_3,\ x_2x_4\}.$$ The ideal $I=(\mathcal{L}(u,v))$ is a lexsegment ideal which is not completely lexsegment and it has linear quotients with respect to the following order of the generators: $u_1=x_2x_4,\ u_2=x_1x_4,\ u_3=x_2x_3,\ u_4=x_1x_3,\ u_5=x_2^2$. We have  $\set(u_1)=\emptyset,\ \set(u_2)=\{2\},\ \set(u_3)=\{4\},\ \set(u_4)=\{2,4\},\ \set(u_5)=\{3,4\}$. Note that, in this case, the integer $l$ that we fix for defining the order $\prec$ is $l=2$. Let $\mathbb{F}_{\bullet}$ be the minimal graded free resolution of $S/I$.

Since $\max\{|\set(w)|\mid w\in\mathcal{L}(u,v)\}=2$, we have $F_i=0$, for all $i\geq4$.

A basis for the $S-$module $F_1$ is $\{f(\emptyset;u_1),\ f(\emptyset;u_2),\ f(\emptyset;u_3),\ f(\emptyset;u_4),\ f(\emptyset;u_5)\}$.
	
A basis for the $S-$module $F_2$ is $$\{f(\{2\};u_2),\ f(\{4\};u_3),\ f(\{2\};u_4),\ f(\{4\};u_4),\ f(\{3\};u_5),\ f(\{4\};u_5)\}.$$

A basis for the $S-$module $F_3$ is $\{f(\{2,4\};u_4),\ f(\{3,4\};u_5)\}$.

We have the minimal graded free resolution $\mathbb{F}_{\bullet}$:
	\[0\rightarrow S(-4)^2\stackrel{\partial_2}{\rightarrow} S(-3)^6\stackrel{\partial_1}{\rightarrow} S(-2)^5\stackrel{\partial_0}{\rightarrow} S\rightarrow S/I\rightarrow 0
\]
where the maps  are
	\[\partial_0(f(\emptyset;u_i))=u_i,\ \mbox{for}\ 1\leq i\leq 5, 
\]
so \[\partial_0=
\left(\begin {array}{ccccc}
			x_2x_4& x_1x_4& x_2x_3& x_1x_3& x_2^2
	\end{array}\right) .\]
\[\begin{array}{lll}
\partial_1(f(\{2\};u_2))&= &x_1f(\emptyset;u_2)-x_2f(\emptyset;u_1),\\
\partial_1(f(\{4\};u_3))&= &x_3f(\emptyset;u_1)-x_4f(\emptyset;u_3),\\
\partial_1(f(\{2\};u_4))&= &x_1f(\emptyset;u_3)-x_2f(\emptyset;u_4),\\
\partial_1(f(\{4\};u_4))&= &x_3f(\emptyset;u_2)-x_4f(\emptyset;u_4),\\
\partial_1(f(\{3\};u_5))&= &x_2f(\emptyset;u_3)-x_3f(\emptyset;u_5),\\
\partial_1(f(\{4\};u_5))&= &x_2f(\emptyset;u_1)-x_4f(\emptyset;u_5),\\
\end{array}\]
so
\[\partial_1=
\left(\begin {array}{cccccc}
			x_1& x_3& 0& 0& 0&x_2\\
			-x_2& 0& 0& x_3& 0&0\\
			0& -x_4& x_1& 0& x_2&0\\
			0& 0& -x_2& -x_4& 0& 0\\
			0& 0& 0& 0& -x_3& -x_4
	\end{array}\right) .\]

	\[\partial_2(f(\{2,4\};u_4))=x_1f(\{4\};u_3)-x_3f(\{2\};u_2)-x_2f(\{4\};u_4)+x_4f(\{2\};u_4),\]
\[\partial_2(f(\{3,4\};u_5))=x_2f(\{4\};u_3)-x_2f(\{3\};u_1)-x_3f(\{4\};u_5)+x_4f(\{3\};u_5)=\]\[=x_2f(\{4\};u_3)-x_3f(\{4\};u_5)+x_4f(\{3\};u_5),
\]
since $\{3\}\nsubseteq\set(u_1)$, so\[\partial_2=
\left(\begin {array}{cc}
			-x_3&0\\
			x_1&x_2\\
			x_4&0\\
			-x_2&0\\
			0&x_4\\
			0&-x_3
	\end{array}\right) .\]

\end{document}